\documentclass[11pt]{amsart}

\usepackage[a4paper,margin=1.15in]{geometry}
\usepackage{amssymb,amsthm,mathtools}
\usepackage[T1]{fontenc}
\usepackage{lmodern}
\usepackage{microtype}
\usepackage[colorlinks=true,linkcolor=blue,citecolor=blue,urlcolor=blue]{hyperref}
\usepackage{enumitem}

\newtheorem{theorem}{Theorem}[section]
\newtheorem{lemma}[theorem]{Lemma}
\newtheorem{corollary}[theorem]{Corollary}

\newtheorem{conjecture}[theorem]{Conjecture}
\theoremstyle{definition}
\newtheorem{definition}[theorem]{Definition}
\theoremstyle{remark}
\newtheorem{remark}[theorem]{Remark}

\numberwithin{equation}{section}

\newcommand{\Id}{\operatorname{Id}}
\newcommand{\Vol}{\operatorname{Vol}}
\newcommand{\rk}{\operatorname{rk}}
\newcommand{\codim}{\operatorname{codim}}
\newcommand{\Pic}{\operatorname{Pic}}
\newcommand{\BC}{\mathrm{BC}}
\newcommand{\cO}{\mathcal O}
\newcommand{\cE}{\mathcal E}
\newcommand{\cF}{\mathcal F}
\newcommand{\cQ}{\mathcal Q}
\newcommand{\cS}{\mathcal S}
\newcommand{\cA}{\mathcal A}
\newcommand{\cG}{\mathcal G}

\title[Mean curvature nonpositivity and rational connectedness]
{Mean curvature nonpositivity and rational connectedness}
\author[Chuanjing Zhang]{Chuanjing Zhang}
\address{School of Mathematics and Statistics, Ningbo University, Ningbo, 315211, P.R. China} \email{zhangchuanjing@nbu.edu.cn}
\author[Xi Zhang]{Xi Zhang}
\address{School of Mathematics and Statistics, Nanjing University of Science and Technology, Nanjing, 210094, P.R. China} \email{mathzx@njust.edu.cn}

\begin{document}

\begin{abstract}
Let $(X,\omega)$ be a compact Gauduchon manifold and let $E$ be a holomorphic vector bundle on $X$. We establish equivalent conditions for the existence of a Hermitian metric on $E$ with nonpositive mean curvature. For a non-uniruled compact K\"ahler manifold $X$ satisfying $\mu_{\omega_0,\max}(T_X)=0$ for some Gauduchon metric $\omega_0$, we show that the corresponding maximal destabilizing subsheaf of $T_X$ splits off holomorphically and is polystable of slope zero with respect to every Gauduchon metric. Its complement, if nonzero, is HN-negative with respect to every such metric. In particular, a compact K\"ahler manifold is not uniruled if and only if , for every Gauduchon metric $\omega$, its tangent bundle admits a Hermitian metric with nonpositive $\omega$-mean curvature. We also obtain an uniruledness criterion that includes the case where the maximal slope of $T_X$ is zero, and a Mumford-type characterization of rational connectedness. Under HN-nonnegativity of $T_X$ for one Gauduchon metric, rational connectedness, the vanishing of all holomorphic covariant tensors, and the condition $\kappa^+(X)=-\infty$ are equivalent.
\end{abstract}

\subjclass[2020]{53C07, 14J60, 32Q15.}
\thanks{The  authors are partially supported by NSF in China No.12431004, 12622103, 12371062 and 12141104. The research was supported by the Fundamental Research Funds for the Central Universities No.30926WZCT202601009 and the Stability Support Program for Basic Research No.WZJC202601001.}
\maketitle

\section{Introduction}

Let $(X,\omega)$ be a connected compact Hermitian manifold of dimension $n\geq 2$ and let $E$ be a holomorphic
vector bundle on $X$. Positivity of $E$ can be measured in two different ways. The algebro-geometric approach uses slopes of coherent subsheaves and quotient sheaves, whereas the differential-geometric approach uses the mean curvature of Hermitian metrics. On compact K\"ahler manifolds, the Donaldson-Uhlenbeck-Yau theorem \cite{Donaldson,UY1986} states that a holomorphic vector bundle is polystable if and only if it admits a Hermitian-Einstein metric. Li and Yau \cite{LY1987} extended this result to compact Gauduchon manifolds; see also \cite{Kobayashi,LT1995}.

 A Hermitian metric $\omega$ on $X$ is called Gauduchon if $\partial\bar\partial\omega^{n-1}=0$. It is known that every conformal class of Hermitian metrics on a connected compact complex manifold of dimension at least two contains a Gauduchon representative, unique up to a positive constant \cite{Gauduchon,Michelsohn}. In dimension one, every Hermitian metric is Gauduchon. Throughout the paper, manifolds are assumed to be connected. Suppose that $\omega$ is Gauduchon. For a torsion-free coherent sheaf $\cE$, its $\omega$-degree and $\omega$-slope are defined by
\begin{equation}
\deg_\omega(\cE):=\int_Xc_1(\det\cE , \hat{H})\wedge
\frac{\omega^{n-1}}{(n-1)!}, \qquad
\mu_\omega(\cE):=\frac{\deg_\omega(\cE)}{\rk\cE},
\end{equation}
where $\hat{H}$ is an arbitrary Hermitian metric on $\det\cE $. We denote by $\mu_{\omega,\max}(\cE)$ and
$\mu_{\omega,\min}(\cE)$ the maximal and minimal slopes of the successive quotients in the Harder-Narasimhan filtration of $E$, respectively. For the existence and uniqueness of this filtration in the Gauduchon setting, see \cite{Bruasse,Kobayashi,LT1995}. We say that $\cE$ is
HN-negative, HN-nonpositive, HN-positive, or HN-nonnegative with respect to $\omega$ if, respectively,
\begin{equation}
\mu_{\omega,\max}(\cE)<0,\quad \mu_{\omega,\max}(\cE)\leq0,\quad \mu_{\omega,\min}(\cE)>0,\quad \mu_{\omega,\min}(\cE)\geq0.
\end{equation}
If $H$ is a smooth Hermitian metric on $E$, let $F_H$ be the curvature of its Chern connection $D_{H}$ and set
\begin{equation}
 \Phi_{H,\omega}:=\sqrt{-1}\,\Lambda_\omega F_H \in C^\infty(X,\operatorname{End}(E)).
\end{equation}
$\Phi_{H,\omega}$ is called the $\omega$-mean curvature of $H$. At every point $x\in X$, the endomorphism $\Phi_{H,\omega}(x)$ is $H$-Hermitian. Write
\begin{equation}
 \lambda_{\max}(\Phi_{H,\omega})(x):=\max_{0\neq v\in E_x}\frac{H\bigl(\Phi_{H,\omega}(x)v,v\bigr)}{H(v,v)}
\end{equation}
for its largest eigenvalue. The metric $H$ has positive, nonnegative, negative, or nonpositive $\omega$-mean curvature according to the corresponding pointwise sign of $\Phi_{H,\omega}$.

According to \cite[Theorem 1.5]{LZZ2015}, for every $a\in\mathbb{R}$, we have
\begin{equation}\label{eq:strict-correspondence}
\mu_{\omega,\max}(E)<a \quad\Longleftrightarrow\quad E\ \text{admits }H\text{ with }\Phi_{H,\omega} <\frac{2\pi a}{\Vol_\omega(X)}\Id_E.
\end{equation}
The corresponding statement for $\mu_{\omega,\min}(E)$ follows by duality. As a consequence, we see that a compact K\"ahler manifold is projective and rationally connected if and only if its tangent bundle admits a Hermitian metric with positive mean curvature \cite[Theorem 1.7]{LZZ2015}.

The strict inequalities in \eqref{eq:strict-correspondence} cannot in general be replaced by non-strict inequalities. A semistable bundle $E$ of slope $a$ admits approximate Hermitian-Einstein metrics \cite{Jacob,NZ2018}. However, it need not admit a Hermitian metric $H$ satisfying $\Phi_{H,\omega}\leq \frac{2\pi a}{\Vol_\omega(X)}\Id_E$. Indeed, since $\mu_\omega(E)=a$, taking the trace and integrating shows that this inequality forces equality everywhere. Such a metric would therefore be Hermitian-Einstein and $E$ would be polystable. A nonsplit extension of stable bundles of slope $a$ is semistable but not polystable, so it admits no such metric. The following theorem gives necessary and sufficient conditions for the existence of a Hermitian metric satisfying this upper bound.

\begin{theorem}\label{thm:metric}
Let $(X,\omega)$ be a compact Gauduchon manifold, let $E$ be a holomorphic vector bundle, and let $a\in\mathbb R$. Then the following conditions are equivalent:
\begin{enumerate}[label=\textup{(\roman*)},ref=\roman*]
\item\label{(i)} There exists a smooth Hermitian metric $H_0$ on $E$ such that
\begin{equation}\label{eq:integral-condition}
\int_X\lambda_{\max}(\Phi_{H_0,\omega})\,dV_\omega\leq2\pi a, \qquad dV_\omega:=\frac{\omega^n}{n!}.
\end{equation}
\item\label{(ii)} There is a holomorphic decomposition
\begin{equation}
E=E_a\oplus E_{<a},
\end{equation}
where $E_a$ is either zero or $\omega$-polystable of slope $a$, and $E_{<a}$ is either zero or satisfies $\mu_{\omega,\max}(E_{<a})<a$. If $\mu_{\omega,\max}(E)=a$, then $E_{a}=\cS$, where $\cS$ is the maximal destabilizing subsheaf.
\item\label{(iii)} There exists a smooth Hermitian metric $H$ on $E$ such that
\begin{equation}\label{eq:pointwise-boundary}
\Phi_{H,\omega}\leq \frac{2\pi a}{\Vol_\omega(X)}\Id_E.
\end{equation}
\end{enumerate}
\end{theorem}
Taking $a=0$ gives a criterion for the existence of a Hermitian metric with nonpositive mean curvature.

We next turn to the tangent bundle. Recall that a compact K\"ahler manifold $X$ is called uniruled if it is dominated by a family of rational curves. Consider the Mumford-type vanishing condition
\begin{equation}\label{eq:tensor-vanishing}
 H^0\bigl(X,(\Omega_X^1)^{\otimes m}\bigr)=0 \qquad\text{for every }m>0.
\end{equation}
Rationally connected projective manifolds satisfy \eqref{eq:tensor-vanishing}; the converse is usually referred to as the Mumford conjecture; see \cite[Remark 1.2]{CDP2015} and \cite{LP2017}. Lazi\'c and Peternell proposed the following conjecture:
\begin{conjecture}\cite[Conjecture 1.2]{LP2017}\label{cLP}
Let $X$ be a compact projective manifold satisfying \eqref{eq:tensor-vanishing}. Then $X$ is uniruled.
\end{conjecture}
Moreover, they showed that it implies the Mumford conjecture. Recently, Ou \cite[Theorem 1.1]{Ou2025} proved that a compact K\"ahler manifold is uniruled if and only if its canonical line bundle is not pseudoeffective. Cao and P\u{a}un \cite[Corollaries 5.2 and 5.3]{CaoPaun2025} established the corresponding positivity results for quotients of the cotangent bundle and for tangent slopes with respect to movable classes. Together with Lamari's criterion \cite{Lamari}, these results imply that $X$ is uniruled if and only if there exists a Gauduchon metric $\omega$ such that $\mu_{\omega,\max}(T_X)>0$. We now consider the case $\mu_{\omega,\max}(T_X)=0$. By using the holomorphic splitting theorem of Loray-Pereira-Touzet \cite[Theorem 5.2]{LPT2018}, we obtain the following theorem.

\begin{theorem}\label{thm2}
Let $X$ be a compact K\"ahler manifold. Assume that $X$ is not uniruled and that there exists a Gauduchon metric $\omega_0$ on $X$ such that
\begin{equation}\label{HNzero}
\mu_{\omega_0,\max}(T_X)=0.
\end{equation}
Then there exist integrable holomorphic subbundles $\cS,\cQ\subset T_X$, with $\cS\neq0$, such that
\begin{equation}\label{eq:split}
 T_X=\cS\oplus\cQ,
\end{equation}
where $\cS$ is the maximal destabilizing subsheaf for $\omega_0$, $c_1^{\BC}(\cS)=0$, and $\cS$ is $\omega$-polystable of slope zero for every Gauduchon metric $\omega$. If $\cQ\neq0$, then $\mu_{\omega,\max}(\cQ)<0$ for every Gauduchon metric $\omega$. Furthermore, a compact K\"ahler manifold $X$ is not uniruled if and only if
for every Gauduchon metric $\omega$ there exists a smooth Hermitian metric $H_\omega$ on $T_X$ such that
\begin{equation}
\Phi_{H_\omega,\omega}\leq0.
\end{equation}
\end{theorem}

The proof first shows that the subsheaf $\cS$ of slope zero has vanishing first Bott-Chern class. The splitting theorem then makes it a holomorphic subbundle. Applying the same argument to the term of a Jordan-H\"older filtration gives polystability for every Gauduchon metric. This also shows that $\cS$ is independent of the chosen Gauduchon metric.

As an application, we obtain the following uniruledness criterion.

\begin{theorem}\label{thm:uniruled}
Let $X$ be a compact K\"ahler manifold satisfying \eqref{eq:tensor-vanishing}. Assume that there exists a Gauduchon metric $\omega_1$ on $X$ such that
\begin{equation}\label{eq:uniruled-hypotheses}
 \mu_{\omega_1,\max}(T_X)\geq0.
\end{equation}
Then $X$ is uniruled.
\end{theorem}

We propose the following conjecture:
\begin{conjecture}\label{c1}
Let $X$ be a compact K\"ahler manifold. If the tangent bundle $T_{X}$ is HN-negative for any Gauduchon metric on $X$, then \eqref{eq:tensor-vanishing} does not hold.
\end{conjecture}
By Theorem \ref{thm:uniruled}, Conjecture~\ref{c1} implies Conjecture \ref{cLP}. The strictly HN-negative case remains open in this argument.

The invariant below was introduced by Campana \cite{Campana}.

\begin{definition}\label{def:kappa-plus}
For a compact complex manifold $X$, set
\begin{equation}
 \kappa^+(X):=\max\bigl\{\kappa(X,\det\cF)\,\big|\,
 p\geq1,\ 0\neq\cF\subset\Omega_X^p\text{ a coherent subsheaf}\bigr\},
\end{equation}
where $\det\cF=(\Lambda^{\rk\cF}\cF)^{**}$ and the maximum of the empty set is $-\infty$.
\end{definition}

In \cite{DPS1996}, Demailly-Peternell-Schneider explicitly formulated the following conjecture after replacing Hermitian
semipositivity of $-K_X$ by nefness. The question was later restated by Campana-Peternell \cite[Problem 6.12(2)]{CP1999}.

\begin{conjecture}[Demailly-Peternell-Schneider]\label{conj:DPS}
Let $X$ be a compact K\"ahler manifold such that $-K_X$ is nef and $\kappa^+(X)=-\infty$. Then $X$ is rationally connected.
\end{conjecture}

We also obtain the following Mumford-type criterion. The important point is that $T_X$ is assumed to be HN-nonnegative with respect to only one Gauduchon metric.

\begin{theorem}\label{thm:mumford}
Let $X$ be a compact K\"ahler manifold. Assume that there exists a Gauduchon metric $\omega_0$ such that
\begin{equation}\label{eq:HN-nonnegative}
\mu_{\omega_0,\min}(T_X)\geq0.
\end{equation}
Then the following conditions are equivalent:
\begin{enumerate}[font=\upshape]
\item\label{(1)} $X$ is projective and rationally connected;
\item\label{(2)} $H^0\bigl(X,(\Omega_X^1)^{\otimes m}\bigr)=0$ for every $m\geq1$;
\item\label{(3)} $\kappa^+(X)=-\infty$.
\end{enumerate}
\end{theorem}

Cao proved that if $-K_X$ is nef on a compact K\"ahler manifold, then $T_X$ is generically $\omega^{n-1}$-semipositive for every K\"ahler class $[\omega]$; equivalently, all Harder-Narasimhan slopes with respect to $[\omega^{n-1}]$ are nonnegative \cite[Theorem 1.2]{Cao2013}. In the same paper, Cao characterizes rational connectedness by the vanishing of holomorphic covariant tensors under the assumption that $-K_X$ is nef \cite[Proposition 1.3]{Cao2013}; Cao also points out the related earlier algebraic result of Q. Zhang \cite[Corollary 1]{Zhang2005}. Thus nefness of $-K_X$ implies the hypothesis \eqref{eq:HN-nonnegative} for every K\"ahler metric. The hypothesis of Theorem \ref{thm:mumford} is therefore formally weaker and requires HN-nonnegativity with respect to only one Gauduchon metric. In particular, the Demailly-Peternell-Schneider conjecture \ref{conj:DPS} follows from Theorem \ref{thm:mumford} and  \cite[Theorem 1.2]{Cao2013}.

The structure theory of manifolds with nef anticanonical bundle has also advanced substantially. Cao and H\"oring established the expected structure theorem for smooth projective manifolds \cite{CaoHoring2019}; Matsumura and Wang extended it to projective klt pairs \cite{MWang2025}; and Matsumura and Wu treated the compact K\"ahler threefold case \cite{MWu2025}. The recent preprint of Matsumura, Wang, Wu and Zhang \cite[Corollary 1.6]{MWWZ2026} proves the full compact K\"ahler structure theorem: if $-K_X$ is nef, then $X$ admits a
locally constant fibration $f:X\longrightarrow Y$ whose fibers are rationally connected and whose base satisfies $c_1(Y)=0$. This also gives a direct structural route to Conjecture \ref{conj:DPS}. Indeed, if $q:=\dim Y>0$, the Beauville-Bogomolov-Yau decomposition implies that $K_Y$ is torsion \cite{Beauville}. Since $f$ is a smooth submersion, the natural map
\begin{equation}
 f^*K_Y=f^*\Omega_Y^q\hookrightarrow\Omega_X^q
\end{equation}
is injective. Hence a torsion line bundle with Kodaira dimension zero occurs as an invertible subsheaf of $\Omega_X^q$, contradicting $\kappa^+(X)=-\infty$. Thus $Y$ is a point and $X$ is rationally connected. Our proof is different: it uses neither nefness of $-K_X$ nor an MRC fibration, and requires $T_X$ to be HN-nonnegative with respect to only one Gauduchon metric.

The paper is organized as follows. In Section 2, we prove Theorem \ref{thm:metric} and show that the pseudoeffective invertible subsheaves considered in Lemma \ref{lem:zero-mass} have vanishing first Bott-Chern class. In Section 3, we prove Theorems \ref{thm2} and \ref{thm:uniruled}. Section 4 is devoted to the proof of Theorem \ref{thm:mumford}.

\section{Criterion for nonpositive mean curvature}

\begin{lemma}\label{lem:GC}
Let $\cS\subset E$ be a saturated coherent subsheaf of rank $s$, and let $\Sigma$ be the analytic subset outside which $\cS$ and $E/\cS$ are locally free. If $\gamma$ is the second fundamental form of $\cS|_{X\setminus\Sigma}$
with respect to $H$, then
\begin{equation}
\begin{split}
2\pi\deg_\omega(\cS)
&=\int_{X\setminus\Sigma}
 \operatorname{tr}\bigl(\pi_{H}^{\cS}\Phi_{H,\omega}|_{\cS}\bigr)\,dV_\omega
-\int_{X\setminus\Sigma}|\gamma|^2\,dV_\omega \\
&\leq s\int_X\lambda_{\max}(\Phi_{H,\omega})\,dV_\omega
-\|\gamma\|_{L^2}^2,\label{eq:GC}
\end{split}
\end{equation}
where $\pi_{H}^{\cS}: E\rightarrow E$ is the orthogonal projection onto $\cS$ with respect to the metric $H$.
\end{lemma}

\begin{proof}
On $X\setminus\Sigma$, the curvature formula for a Hermitian holomorphic
subbundle is
\begin{equation}
F_{\cS}=\pi_{H}^{\cS}F_H|_{\cS}-\gamma^*\wedge\gamma
\end{equation}
\cite[Chapter V, \S6]{Kobayashi}. The equality in \eqref{eq:GC} follows by resolving $\cS\subset E$ to a subbundle on a modification and integrating the trace of the curvature formula against the pullback of $\omega^{n-1}/(n-1)!$. The Gauduchon condition gives the degree term, and the exceptional divisors contribute zero because their images have codimension at least two; see \cite{BS1994,LT1995} for background. Finally,
\begin{equation}
 \operatorname{tr}\bigl(\pi_{H}^{\cS}\Phi_{H,\omega}|_{\cS}\bigr) \leq s\lambda_{\max}(\Phi_{H,\omega})
\end{equation}
pointwise.
\end{proof}

\begin{corollary}\label{cor:splitting}
Let $a\in\mathbb R$ and assume that $H_0$ satisfies \eqref{eq:integral-condition}. Then $\mu_{\omega,\max}(E)\leq a$. If a saturated subsheaf $\cS\subset E$ of rank $s$ satisfies $\mu_\omega(\cS)=a$, then $\cS$ is a holomorphic subbundle and $E=\cS\oplus\cS^{\perp_{H_0}} $holomorphically.
\end{corollary}

\begin{proof}
For every saturated subsheaf $\cF\subset E$ of rank $r$, Lemma \ref{lem:GC} and \eqref{eq:integral-condition} give
\begin{equation}
2\pi\deg_\omega(\cF) \leq r\int_X\lambda_{\max}(\Phi_{H_0,\omega})\,dV_\omega \leq2\pi r a.
\end{equation}
Hence $\mu_{\omega,\max}(E)\leq a$. If $\mu_\omega(\cS)=a$, then $\deg_\omega(\cS)=sa$ and \eqref{eq:GC} yields
\begin{equation}
2\pi sa \leq s\int_X\lambda_{\max}(\Phi_{H_0,\omega})\,dV_\omega -\|\gamma\|_{L^2}^2 \leq2\pi sa-\|\gamma\|_{L^2}^2.
\end{equation}
Thus $\gamma=0$ (and in fact equality holds in \eqref{eq:integral-condition}). The orthogonal projection $\pi_{H_{0}}^{\cS}$ is therefore parallel on $X\setminus\Sigma$ and in particular holomorphic there. Its matrix coefficients are locally bounded, because $\pi_{H_{0}}^{S}$ is an orthogonal projection. Since $\codim_{\mathbb{C}} \Sigma \geq 2$, the Riemann extension theorem gives a holomorphic extension of $\pi_{H_{0}}^{S}$ to an endomorphism of $E$ on $X$. The identity $(\pi_{H_{0}}^{\cS})^2=\pi_{H_{0}}^{\cS}$ extends as well, so its image and kernel are holomorphic subbundles. Since $\cS$ is saturated, the image of the extended projection is $\cS$.
\end{proof}

The following lemma avoids a circular argument that could otherwise arise when the maximal destabilizing subsheaf of slope $a$ is the whole bundle.

\begin{lemma}\label{lem:rigid}
Let $a\in\mathbb R$. Assume that a holomorphic vector bundle $E$ admits a metric $H_0$ satisfying \eqref{eq:integral-condition}. If $E$ is $\omega$-semistable with $\mu_\omega(E)=a$, then $E$ is $\omega$-polystable of slope $a$.
\end{lemma}

\begin{proof}
We argue by induction on $\rk E$. If $E$ is stable there is nothing to prove. Otherwise, choose the first nonzero term $\cS$ in a Jordan-H\"older filtration of $E$. It may be chosen saturated; it is stable of slope $a$ \cite{Kobayashi,LT1995}. By Corollary \ref{cor:splitting}, $\cS$ is a holomorphic direct factor,
\begin{equation}
 E=\cS\oplus E'.
\end{equation}
The quotient $E'$ is semistable of slope $a$. Moreover, the induced metric on $E'$ still satisfies the same integral hypothesis: the vanishing of the second fundamental form makes the mean curvature block diagonal and
\begin{equation}
\lambda_{\max}(\Phi_{H_0|E',\omega}) \leq\lambda_{\max}(\Phi_{H_0,\omega})
\end{equation}
pointwise. The induction hypothesis applies to $E'$, so $E'$ is polystable of slope $a$. Therefore $E$ is polystable of slope $a$.
\end{proof}

\begin{proof}[Proof of Theorem \ref{thm:metric}]
We prove $\textup{(\ref{(i)})}\Rightarrow \textup{(\ref{(ii)})}\Rightarrow \textup{(\ref{(iii)})}\Rightarrow \textup{(\ref{(i)})}$.

Assume (\ref{(i)}). By Corollary \ref{cor:splitting}, $\mu_{\omega,\max}(E)\leq a$. If $\mu_{\omega,\max}(E)<a$, we simply take $E_a=0$, $E_{<a}=E$. Suppose now that $\mu_{\omega,\max}(E)=a$, and let $\cS$ be the maximal
destabilizing subsheaf, i.e., the first Harder-Narasimhan term. Then $\cS$ is saturated and semistable of slope $a$, while every HN slope of $\cQ:=E/\cS$ is strictly smaller than $a$ unless $\cQ=0$. Since $\mu_\omega(\cS)=a$, Corollary \ref{cor:splitting} gives a holomorphic orthogonal splitting $E=\cS\oplus\cQ$. By Lemma \ref{lem:rigid}, $\cS$ is polystable of slope $a$. Thus (\ref{(ii)}) holds with $E_a=\cS$ and $E_{<a}=\cQ$. Furthermore, we denote the $s$ eigenvalues of the mean curvature $\Phi_{H_0|E_{a},\omega}$  by $\lambda_1$, $\lambda_2$, $\cdots$, $\lambda_s$, listed in nonincreasing order, where $H_0|E_{a}$ is the induced metric on $E_{a}$ with respect to $H_{0}$. We have
\begin{equation}\label{ei}
\lambda_{1}=\lambda_{\max}(\Phi_{H_0|E_{a},\omega}) \leq\lambda_{\max}(\Phi_{H_0,\omega}).
\end{equation}
By (\ref{eq:GC}), we obtain
\begin{equation}
\begin{split}
2\pi s a&= 2\pi s \mu_{\omega }(E_{a})=2\pi\deg_\omega(\cS)=\int_{X}
 \operatorname{tr}\bigl(\Phi_{H_0|E_{a},\omega})\,dV_\omega\\
&=\int_{X}
 (\sum_{i=1}^{s}\lambda_{i})\,dV_\omega \leq s\int_X\lambda_{\max}(\Phi_{H_{0},\omega})\,dV_\omega\leq 2\pi s a
\label{eq:GC2}
\end{split}
\end{equation}
and 
\begin{equation}
\lambda_{1}=\cdots =\lambda_{s}=\lambda_{\max}(\Phi_{H_0|E_{a},\omega}) =\lambda_{\max}(\Phi_{H_0,\omega})\Id_{E_{a}}.
\end{equation}
Then, the induced metric $H_0|E_{a}$ is a Hermitian-Einstein metric on $E_{a}$, that is,
\begin{equation}
\Phi_{H_0|E_{a},\omega}=\lambda_{\max}(\Phi_{H_0,\omega})\Id_{E_{a}}.
\end{equation}
If $\Phi_{H_{0},\omega}\leq \frac{2\pi a}{\Vol_\omega(X)}\Id_E$, i.e., $\lambda_{\max}(\Phi_{H_0,\omega})\leq \frac{2\pi a}{\Vol_\omega(X)}$, then (\ref{ei}) and (\ref{eq:GC2}) imply that
\begin{equation}
\Phi_{H_{0}|_{E_{a}},\omega}= \frac{2\pi a}{\Vol_\omega(X)}\Id_{E_{a}}.
\end{equation}

Assume (\ref{(ii)}). If $E_a\neq0$, then the Kobayashi-Hitchin correspondence on compact Gauduchon manifolds gives a Hermitian-Einstein metric $H_a$ satisfying
\begin{equation}
\Phi_{H_a,\omega}=\frac{2\pi a}{\Vol_\omega(X)}\Id_{E_a}
\end{equation}
\cite{LY1987}. If $E_{<a}\neq0$, then $\mu_{\omega,\max}(E_{<a})<a$, and \eqref{eq:strict-correspondence} gives a metric $H_{<a}$ satisfying
\begin{equation}
 \Phi_{H_{<a},\omega}
 <\frac{2\pi a}{\Vol_\omega(X)}\Id_{E_{<a}}.
\end{equation}
Taking the direct-sum metric on the nonzero summands gives a metric $H$ on $E$ satisfying \eqref{eq:pointwise-boundary}. This proves (\ref{(iii)}).

Finally, (\ref{(iii)}) immediately implies
\begin{equation}
\lambda_{\max}(\Phi_{H,\omega}) \leq\frac{2\pi a}{\Vol_\omega(X)}
\end{equation}
pointwise. Integration over $X$ yields (\ref{(i)}).
\end{proof}

\begin{remark}\label{rem:strictness}
The integral hypothesis \eqref{eq:integral-condition} is stronger than the numerical inequality $\mu_{\omega,\max}(E)\leq a$. The semistable but nonpolystable bundles of slope $a$ discussed in the introduction satisfy $\mu_{\omega,\max}(E)= a$ but by Theorem \ref{thm:metric}, admit no Hermitian metric satisfying \eqref{eq:integral-condition}. Thus the integral condition detects precisely the splitting obstruction when $\mu_{\omega,\max}(E)=a$. When $a=0$, Theorem \ref{thm:metric} says that the integral nonpositivity condition is equivalent to a splitting into a polystable summand of slope zero and an HN-negative summand, and also equivalent to the existence of a metric with nonpositive mean curvature.
\end{remark}

The next observation is the key simplification in the proof of Theorem \ref{thm:mumford}. It requires HN-nonnegativity only for one Gauduchon metric.

\begin{lemma}\label{lem:dual-degree}
Let $(X,\omega)$ be a compact Gauduchon manifold and let $E$ be a holomorphic vector bundle with $\mu_{\omega,\min}(E)\geq0$. Then, for every $m\geq1$ and every invertible subsheaf $L\subset(E^*)^{\otimes m}$,
\begin{equation}
\deg_\omega(L)\leq0.
\end{equation}
The same conclusion holds for every invertible subsheaf of a holomorphic subbundle of $(E^*)^{\otimes m}$.
\end{lemma}

\begin{proof}
After replacing $L$ by its saturation, which can only increase its degree, we may assume that $L$ is saturated. By duality for Harder-Narasimhan slopes \cite{Kobayashi,LT1995},
\begin{equation}
 \mu_{\omega,\max}(E^*)=-\mu_{\omega,\min}(E)\leq0.
\end{equation}
Fix $\varepsilon>0$. By \eqref{eq:strict-correspondence}, there exists a Hermitian metric $H_\varepsilon$ on $E^*$ such that
\begin{equation}
 \Phi_{H_\varepsilon,\omega}<\varepsilon\Id_{E^*}.
\end{equation}
The induced metric on $(E^*)^{\otimes m}$ satisfies
\begin{equation}
\Phi_{H_\varepsilon^{\otimes m},\omega}<m\varepsilon\Id.
\end{equation}
Applying Lemma~\ref{lem:GC} to $L$ derives
\begin{equation}
 2\pi\deg_\omega(L)\leq m\varepsilon\Vol_\omega(X).
\end{equation}
Letting $\varepsilon\rightarrow0$ yields $\deg_\omega(L)\leq0$. The assertion for a subbundle follows by composition with the inclusion into $(E^*)^{\otimes m}$.
\end{proof}

\begin{lemma}\label{lem:zero-mass}
In the setting of Lemma \ref{lem:dual-degree}, assume in addition that $L$ is pseudoeffective. Then
\begin{equation}
 c_1^{\BC}(L)=0.
\end{equation}
\end{lemma}

\begin{proof}
By Lemma \ref{lem:dual-degree}, $\deg_\omega(L)\leq0$. Since $L$ is pseudoeffective, there is a closed positive $(1,1)$-current $T\in c_1^{\BC}(L)$ \cite{Demailly2012}. Therefore
\begin{equation}
 \deg_\omega(L)=\int_XT\wedge\frac{\omega^{n-1}}{(n-1)!}\geq0.
\end{equation}
Hence the integral is zero. Since $T\geq0$ and $\omega^{n-1}$ is a smooth strictly positive $(n-1,n-1)$-form, a nonzero positive current has strictly positive mass against $\omega^{n-1}$ \cite[Chapter III]{Demailly2012}. Consequently $T=0$, and therefore $c_1^{\BC}(L)=0$.
\end{proof}

\begin{corollary}\label{cor:torsion}
Let $X$ be compact K\"ahler manifold with $H^0(X,\Omega_X^1)=0$, and let $E$ be a holomorphic vector bundle which is HN-nonnegative with respect to a Gauduchon metric $\omega$. If a pseudoeffective invertible subsheaf $L\subset(E^*)^{\otimes m}$ exists for some $m>0$, then $L$ is torsion. More precisely, there is $N>0$ such that $L^{\otimes N}\simeq\cO_X$, and hence
\begin{equation}
H^0\bigl(X,(E^*)^{\otimes mN}\bigr)\neq0.
\end{equation}
\end{corollary}

\begin{proof}
By Lemma \ref{lem:zero-mass}, $c_1^{\BC}(L)=0$, hence $c_1(L)=0$ in $H^2(X,\mathbb R)$ \cite{Demailly2012}. Thus the integral first Chern class of $L$ is torsion. Since $X$ is K\"ahler, Hodge symmetry gives $H^1(X,\cO_X)=H^{0,1}(X)=0$, hence $\Pic^0(X)=0$ \cite{Voisin}. Thus, for some $N>0$, $L^{\otimes N}$ has zero integral Chern class and belongs to $\Pic^0(X)$; therefore $L^{\otimes N}\simeq\cO_X$. Tensoring the inclusion $L\subset(E^*)^{\otimes m}$ $N$ times yields the required nonzero section.
\end{proof}

\section{HN-nonpositivity and uniruledness}

Theorem 1.1 in \cite{Ou2025} and Corollary 5.3 in \cite{CaoPaun2025} imply the following lemma.

\begin{lemma}\label{lem:generic}
Let $X$ be a compact K\"ahler manifold. Then $X$ is not uniruled if and only if
\begin{equation}
\mu_{\omega,\max}(T_X)\leq0
\end{equation}
for every Gauduchon metric $\omega$ on $X$.
\end{lemma}

\begin{proof}
Suppose first that $X$ is not uniruled. By \cite[Theorem 1.1]{Ou2025}, $K_X$ is pseudoeffective. Let $\cF\subset T_X$ be a nonzero saturated subsheaf. The image $\cE$ of the restriction morphism $\Omega_X^1\longrightarrow\cF^*$ is a torsion-free quotient of $\Omega_X^1$. Outside an analytic subset of codimension at least two, $\cF$ and $T_X/\cF$ are locally free, and this morphism is surjective. Consequently
\begin{equation}
\det\cE\simeq(\det\cF)^*.
\end{equation}
By \cite[Corollary 5.2]{CaoPaun2025}, $\det\cE$ is pseudoeffective. It follows that $\deg_\omega(\cF)\leq0$ for every Gauduchon metric $\omega$. Taking the maximal slope yields the claimed inequality.

Conversely, the inequality for the whole tangent bundle gives
\begin{equation}
 \deg_\omega K_X=-\deg_\omega T_X\geq0 \qquad\text{for every Gauduchon }\omega.
\end{equation}
Lamari's criterion \cite[Lemme 3.3]{Lamari}, together with Michelsohn's root construction \cite{Michelsohn}, implies that $K_X$ is pseudoeffective. Ou's theorem then shows that $X$ is not uniruled.
\end{proof}

\begin{lemma}\label{lem:variation}
Let $X$ be a compact K\"ahler manifold and suppose that $\mu_{\omega,\max}(T_X)\leq0$ for every Gauduchon $\omega$. If a saturated subsheaf $\cF\subset T_X$ has
$\deg_{\omega_0}(\cF)=0$ for one Gauduchon metric $\omega_0$, then
\begin{equation}
c_1^{\BC}(\cF)=0, \qquad \deg_\omega(\cF)=0 \quad\text{for every Gauduchon }\omega.
\end{equation}
Moreover, $\cF$ is $\omega$-semistable of slope zero for every such metric.
\end{lemma}

\begin{proof}
Keep the subsheaf $\cF$ fixed and set $L=(\det\cF)^*$. For every Gauduchon metric $\omega$, one has
\begin{equation}
\deg_\omega L=-\deg_\omega\cF \geq-\rk(\cF)\mu_{\omega,\max}(T_X)\geq0.
\end{equation}
Lamari's criterion shows that $L$ is pseudoeffective. By the same argument as in Lemma \ref{lem:zero-mass}, we deduce $c_1^{\BC}(\cF)=-c_1^{\BC}(L)=0$. The degree of $\cF$ therefore vanishes for every Gauduchon metric.
Every nonzero subsheaf of $\cF$ is also a subsheaf of $T_X$, so its slope is nonpositive. This proves semistability.
\end{proof}

We use the following analytic part of the splitting theorem of Loray-Pereira-Touzet.

\begin{lemma}[Loray-Pereira-Touzet]\label{lem:LPT}
Let $X$ be a compact K\"ahler manifold with $K_X$ pseudoeffective. If $\cF\subset T_X$ is a nonzero saturated subsheaf with $c_1(\cF)=0$ in $H^2(X,\mathbb R)$, then $\cF$ is an integrable holomorphic subbundle. There is an integrable holomorphic subbundle $\cG$ such that
\begin{equation}
 T_X=\cF\oplus\cG.
\end{equation}
When $\rk\cF=\dim X$, this means $\cF=T_X$ and $\cG=0$.
\end{lemma}

For $\rk\cF<\dim X$, this is \cite[Theorem 5.2]{LPT2018}; for full rank, it follows from saturation. Only the regularity and splitting assertions are used. In particular, the extra projective hypothesis for the torsion assertion in that theorem is not needed here.

\begin{lemma}\label{lem:universal-polystability}
Under the hypotheses of Lemma \ref{lem:variation}, every nonzero saturated subsheaf $\cF\subset T_X$ with $c_1^{\BC}(\cF)=0$ is a holomorphic subbundle and is $\omega$-polystable of slope zero for every Gauduchon metric $\omega$.
\end{lemma}

\begin{proof}
The hypothesis for the whole tangent bundle and Lamari's criterion give pseudoeffectivity of $K_X$, as in Lemma \ref{lem:generic}. The natural map from Bott-Chern to real de Rham cohomology sends $c_1^{\BC}(\cF)$ to $c_1(\cF)$. Thus Lemma \ref{lem:LPT} makes $\cF$ a holomorphic direct summand of $T_X$.

Fix a Gauduchon metric $\omega$. It is easy to see that the bundle $\cF$ is $\omega$-semistable of slope zero. We prove its polystability by induction on its rank. If $\cF$ is stable, there is nothing to prove. Otherwise, let $\cA\subset\cF$ be the first nonzero stable term in a Jordan-H\"older filtration. It is saturated in $\cF$ and has slope zero. Since $\cF$ is a direct factor of $T_X$, the sheaf $T_X/\cA$ is torsion-free. Hence $\cA$ is also saturated in $T_X$. Lemma \ref{lem:variation} means $c_1^{\BC}(\cA)=0$, and from Lemma \ref{lem:LPT} there is a holomorphic retraction $p:T_X\longrightarrow\cA$.

Because $\cA\subset\cF$, the restriction of $p$ to $\cF$ is still a retraction. Therefore
\begin{equation}
\cF=\cA\oplus\cF', \qquad \cF':=\ker(p|_{\cF}).
\end{equation}
The bundle $\cF'$ is saturated in $T_X$ and has $c_1^{\BC}(\cF')=0$. Its rank is smaller than that of $\cF$. By induction, it is $\omega$-polystable of slope zero. Since $\cA$ is stable of slope zero, $\cF$ is polystable. The metric $\omega$ was arbitrary, which proves the assertion.
\end{proof}

\begin{proof}[Proof of Theorem \ref{thm2}]
Assume first that $X$ is not uniruled and $\mu_{\omega_0,\max}(T_X)=0$. Let $\cS$ be the first nonzero Harder-Narasimhan term of $T_X$ with respect to $\omega_0$. If $T_X$ is semistable, take $\cS=T_X$. Then $\cS$ is saturated, nonzero, and has $\deg_{\omega_0}\cS=0$. From Lemmas \ref{lem:generic} and \ref{lem:variation}, one can see
\begin{equation}
c_1^{\BC}(\cS)=0.
\end{equation}
Since $K_X$ is pseudoeffective, Lemma \ref{lem:LPT} gives integrable holomorphic subbundles $\cS,\cQ$ with
$T_X=\cS\oplus\cQ$. Lemma \ref{lem:universal-polystability} tells us that $\cS$ is polystable of slope zero for every Gauduchon metric.

If $\cQ\neq0$, the definition of the first HN term means $\mu_{\omega_0,\max}(\cQ)<0$. For every Gauduchon metric $\omega$, the inclusion $\cQ\subset T_X$ implies $\mu_{\omega,\max}(\cQ)\leq0$. Suppose that equality holds for
some metric $\omega_1$. Let $\cA\subset\cQ$ be the first nonzero HN term for $\omega_1$. It is saturated in $\cQ$ and hence in $T_X$, and $\deg_{\omega_1}\cA=0$. Lemma \ref{lem:variation} yields $\deg_{\omega_0}\cA=0$, contradicting $\mu_{\omega_0,\max}(\cQ)<0$. Therefore the same complement $\cQ$ is strictly HN-negative for every Gauduchon metric. In particular, $\cS$ is the first HN term of $T_X$ for every such metric.

We now prove the final equivalence, without assuming \eqref{HNzero}. If $X$ is not uniruled, fix any Gauduchon metric $\omega$. Because of Lemma \ref{lem:generic}, $\mu_{\omega,\max}(T_X)\leq0$. If it is strictly negative, the equivalence \eqref{eq:strict-correspondence} gives a metric with negative mean curvature. If it is zero, the splitting just proved satisfies condition (\ref{(ii)}) of Theorem \ref{thm:metric} with $a=0$. That theorem gives a metric $H_\omega$ with $\Phi_{H_\omega,\omega}\leq0$. Conversely, suppose such a metric exists for every Gauduchon
$\omega$. From Lemma \ref{lem:GC}, it follows that $\mu_{\omega,\max}(T_X)\leq0$ for every $\omega$. Lemma \ref{lem:generic} then shows that $X$ is not uniruled.
\end{proof}

\begin{remark}
The subbundle $\cS$ in Theorem \ref{thm2} is independent of the Gauduchon metric. One can also choose a fixed decomposition $\cS=\bigoplus_j\cS_j$ whose nonzero summands have $c_1^{\BC}(\cS_j)=0$ and are stable for every Gauduchon metric. Indeed, first decompose $\cS$ into stable factors for $\omega_0$ using the preceding proof. Each factor is semistable of slope zero for every other metric. If it were not stable for such a metric,
a proper saturated subsheaf of slope zero would, by Lemma \ref{lem:variation}, also have degree zero for $\omega_0$,
contradicting its $\omega_0$-stability.
\end{remark}

\begin{proof}[Proof of Theorem \ref{thm:uniruled}]
Suppose that $X$ is not uniruled. Due to Lemma \ref{lem:generic}, $\mu_{\omega,\max}(T_X)\leq0$ for every Gauduchon metric. The assumption \eqref{eq:uniruled-hypotheses} therefore gives
\begin{equation}
\mu_{\omega_1,\max}(T_X)=0.
\end{equation}
According to Theorem \ref{thm2}, there exist holomorphic subbundles $\cS,\cQ\subset T_X$, with $\cS\neq 0$, such that $T_X=\cS\oplus\cQ$, where $\cS$ is the maximal destabilizing subsheaf for $\omega_1$ and  $c_1^{\BC}(\cS)=0$. Hence $c_1(\det \cS)=0$ in $H^2(X,\mathbb R)$. The decomposition $T_X=\cS\oplus\cQ$ induces a holomorphic inclusion
\begin{equation}
L:=(\det\cS)^*\hookrightarrow\Omega_X^r,
\end{equation}
where $r=\rk\cS>0$. By the assumption \eqref{eq:tensor-vanishing}, $H^1(X,\cO_X)=H^{0,1}(X)=0$, hence $\Pic^0(X)=0$. Therefore, for some $N>0$, $L^{\otimes N}\simeq\cO_X$. Taking the $N$-th tensor power of the inclusion $L\subset \Omega_X^r$ yields a nonzero holomorphic section of $(\Omega_X^r)^{\otimes N}$. Since $(\Omega_X^r)^{\otimes N} \hookrightarrow (\Omega_X^1)^{\otimes rN}$, this contradicts \eqref{eq:tensor-vanishing}. Hence $X$ is uniruled.
\end{proof}

\section{HN-nonnegativity and rational connectedness}

\begin{lemma}\label{lem:tensor-kappa}
If a compact complex manifold $X$ satisfies \eqref{eq:tensor-vanishing}, then
\begin{equation}
 \kappa^+(X)=-\infty.
\end{equation}
\end{lemma}

\begin{proof}
Suppose that $0\neq\cF\subset\Omega_X^p$ has rank $r$ and $\kappa(X,\det\cF)\geq0$. Taking the \(r\)-th exterior power of the inclusion \(F\hookrightarrow\Omega_X^p\) and then the double dual yields an inclusion of sheaves
\begin{equation}
\det\cF\hookrightarrow\Lambda^r\Omega_X^p \hookrightarrow(\Omega_X^1)^{\otimes pr}.
\end{equation}
Since $\kappa(X,\det\cF)\geq0$, for some $N>0$ there is a nonzero section of $(\det\cF)^{\otimes N}$. Tensoring the displayed inclusion $N$ times gives a nonzero section of $(\Omega_X^1)^{\otimes prN}$, contradicting \eqref{eq:tensor-vanishing}.
\end{proof}

\begin{lemma}\label{lem:kappa-projective}
Let $X$ be a compact K\"ahler manifold with $\kappa^+(X)=-\infty$. Then
\begin{equation}
 H^0(X,\Omega_X^p)=0\qquad\text{for every }p\geq1.
\end{equation}
In particular, $X$ is projective and $\Pic^0(X)=0$.
\end{lemma}

\begin{proof}
If $0\neq s\in H^0(X,\Omega_X^p)$, the image of $\cO_X\xrightarrow{s}\Omega_X^p$ is a rank-one coherent subsheaf. After saturation, its determinant is an effective line bundle, hence has nonnegative Kodaira dimension. This contradicts $\kappa^+(X)=-\infty$.

In particular, $H^{1,0}(X)=H^{2,0}(X)=0$. The latter equality and Hodge decomposition imply $H^2(X,\mathbb R)=H^{1,1}(X,\mathbb R)$. Since the K\"ahler cone is open, it contains a rational class; after multiplication by a positive integer this gives an integral K\"ahler class. Kodaira's embedding theorem therefore shows that $X$ is projective \cite{Voisin}. Finally, Hodge symmetry yields $H^1(X,\cO_X)=H^{0,1}(X)=0$, hence $\Pic^0(X)=0$.
\end{proof}

\begin{lemma}\label{lem:psef-torsion}
Let $X$ be a compact K\"ahler manifold and suppose that, for some Gauduchon metric $\omega_0$,
\begin{equation}
 \mu_{\omega_0,\min}(T_X)\geq0.
\end{equation}
Assume moreover that $H^0(X,\Omega_X^1)=0$. If, for some $p\geq1$, there exists a pseudoeffective invertible subsheaf $L\subset\Omega_X^p$, then, after saturation, $L$ is torsion. In particular, $\kappa(X,L)=0$.
\end{lemma}

\begin{proof}
After saturation, $L$ remains pseudoeffective and the inclusion extends to
\begin{equation}
L\hookrightarrow\Omega_X^p\hookrightarrow(\Omega_X^1)^{\otimes p}.
\end{equation}
Apply Lemma \ref{lem:zero-mass} with $E=T_X$ to obtain $c_1^{\BC}(L)=0$. Corollary \ref{cor:torsion} then shows that $L$ is torsion. A torsion line bundle has Kodaira dimension zero.
\end{proof}

\begin{proof}[Proof of Theorem \ref{thm:mumford}]
$(\ref{(1)})\Rightarrow(\ref{(2)})$ is standard for rationally connected projective manifolds \cite{CDP2015}. The implication $(\ref{(2)})\Rightarrow(\ref{(3)})$ is Lemma \ref{lem:tensor-kappa}. It remains to prove $(\ref{(3)})\Rightarrow(\ref{(1)})$. Assume $\kappa^+(X)=-\infty$. From Lemma \ref{lem:kappa-projective}, we know that $X$ is projective and $\Pic^0(X)=0$. Suppose, for contradiction, that $X$ is not rationally connected. According to the Campana-Demailly-Peternell criterion \cite[Criterion 1.1]{CDP2015}, there exist $p\geq1$ and a pseudoeffective invertible subsheaf $L\subset\Omega_X^p$. By Lemma \ref{lem:psef-torsion}, after saturation $L$ is torsion. Hence $\kappa(X,L)=0$. Since $L\subset\Omega_X^p$, Definition \ref{def:kappa-plus} gives $\kappa^+(X)\geq\kappa(X,L)=0$, a contradiction. Therefore $X$ is rationally connected.
\end{proof}

\end{document}